\documentclass[12pt]{amsart}
\usepackage[T1]{fontenc}
\usepackage[english]{babel}
\usepackage[english]{babel}
\usepackage{amsmath, amssymb, amsthm, amscd,color,comment}
\usepackage{cancel}
\usepackage{cite}
\usepackage{alltt}
\usepackage[dvipsnames]{xcolor}
\usepackage{array}
\usepackage[small,bf,labelsep=period]{caption}
\usepackage{mathtools}
\usepackage{hyperref}

\usepackage{enumerate}
\newtheorem{theorem}{Theorem}[section]
\newtheorem{definition}[theorem]{Definition}
\newtheorem{example}[theorem]{Example}

\newtheorem{lemma}[theorem]{Lemma}
\newtheorem{corollary}[theorem]{Corollary}
\newtheorem{proposition}[theorem]{Proposition}

\DeclarePairedDelimiter\ceil{\lceil}{\rceil}
\DeclarePairedDelimiter\floor{\lfloor}{\rfloor}

\def\la{\lambda}

\def\Q{{\bf Q}}

\def\N{\mathbb N}
\def\R{\mathbb R}
\def\Z{\mathbb Z}

\def\cF{\mathcal F}

\def\cL{\mathcal L}

\def\cP{\mathcal P}
\def\cM{\mathcal M}

\def\cG{\mathcal G}

\def\cX{\mathcal X}
\def\cY{\mathcal Y}

\def\fqs{\mathbb F_{q^s}}

\def\fq{\mathbb F_q}
\def\fp{\mathbb F_p}

\def\Div{{\rm Div}}

\def\deg{{\rm deg}}

\def\lub{{\rm lub}}
\def\glb{{\rm glb}}

\def\negalpha{\text{\boldmath$\alpha$}}

\def\neg1{\text{\boldmath$1$}}
\def\nege{\text{\boldmath$e$}}
\def\negbeta{\text{\boldmath$\beta$}}

\def\neg1{\text{\boldmath$1$}}

\newcommand{\al}{\alpha}
\newcommand{\be}{\beta}

\begin{document}

\title[On generalized Weierstrass semigroups]{On generalized
Weierstrass semigroups in linearized function fields}




\author{Erik Mendoza}

\address{Instituto de Matemática, Universidade Federal do Rio de Janeiro, Cidade Universitária, CEP 21941-909, Rio de Janeiro, Brazil}
\email{erik@im.ufrj.br}

\begin{abstract}
In this article, using the notion of discrepancies, we study the generalized Weierstrass semigroup $\widehat{H}(\mathbf{Q})$, where $\mathbf{Q}$ is an $n$-tuple of distinct totally ramified places of degree one in a linearized function field. As a consequence, we characterize and explicitly determine the sets of absolute maximal elements $\widehat{\Gamma}(\mathbf{Q})$ and relative maximal elements $\widehat{\Lambda}(\mathbf{Q})$, generalizing the existing results in the literature. Finally, we apply our results to some classes of algebraic curves.
\end{abstract}

\maketitle

\section{Introduction}

Let $\fq$ be a finite field with $q$ elements, $\cF$ a function field of one variable over $\fq$, and $\Q=(Q_1, \dots, Q_n)$ an $n$-tuple of distinct rational places in $\cF$. In recent decades, Weierstrass semigroups $H(\Q)$ and related objects, such as the set of gaps $G(\Q)$ and the set of pure gaps $G_0(\Q)$, have been intensely studied due to their applications to other areas. For instance, these objects play a fundamental role in coding theory, particularly in the construction of algebraic geometry (AG) codes with good parameters, see e.g. \cite{BQZ2018, CT2005, GKL1993, KNT2020}. 

The study of the Weierstrass semigruop $H(\Q)$ and their explicit determination are classic problems in algebraic geometry. For the case $n=1$, the Weierstrass semigroup was determined for several specific function fields, such as the Suzuki function field \cite{BMZ2020} and the Skabelund function field \cite{BLM2021}, as well as for families of function fields such as Kummer extensions \cite{ABQ2019, CMS2025, CMQ2024} and linearized function fields \cite{ZZ2026}, which include many of the most important function fields, such as the Hermitian, Norm-Trace, Artin-Schreier, and $GK$ function fields, among others.

The general case for several distinct rational places was first studied in \cite{CT2005}. To simplify its study, Matthews \cite{G2004} introduced the minimal generating set $\Gamma(\mathbf{Q})$, a set that completely determines $H(\mathbf{Q})$. Initially applied to Hermitian \cite{G2004} and Norm-trace \cite{GD2010} function fields, this concept was used to determine $H(\mathbf{Q})$ on several other families of function fields, such as the $GK$ function field \cite{CT2018gk}, function fields defined by curves of the type $f(y)=g(x)$ \cite{CT2018}, Kummer extensions \cite{YH2017, BQZ2018}, and linearized function fileds \cite{ZZ2026}.

The notion of the generalized Weierstrass semigroup $\widehat{H}(\mathbf{Q})$ was introduced by Delgado \cite{D1990} for function fields over algebraically closed fields, and subsequently studied by Beelen and Tuta\c{s} \cite{BT2006} over finite fields. In this context, the sets of absolute maximal elements $\widehat{\Gamma}(\mathbf{Q})$ and relative maximal elements $\widehat{\Lambda}(\mathbf{Q})$ are key objects. The determination of these sets allows us to completely obtain the generalized Weierstrass semigroup $\widehat{H}(\mathbf{Q})$ (and consequently $H(\mathbf{Q})$), the gap set $G(\mathbf{Q})$, and the pure gap set $G_0(\mathbf{Q})$, see \cite{MTT2019, TT2019, CMT2024, CMT2025}.
These sets were determined for several specific function fields, see e.g. \cite{TT2019, MT2023, CMT2025}, and recently, they were determined for arbitrary Kummer extensions \cite{CMT2026}.

Let $s, m$ be positive integers, $q$ be a power of a prime $p$, and $K=\fqs$.  
Consider the linearized polynomial  
$$
L(y) = \sum_{k=0}^{m} a_k y^{q^{k}} \in K[y],
$$  
where $m\leq s$, $a_m \neq 0$, $a_0 \neq 0$, and assume all its roots lie in $K$. Let $h(x)\in K(x)$ be a  rational function and consider the algebraic curve over $K$ defined by the affine equation
\begin{equation*}
	\cX: \quad L(y)=h(x).
\end{equation*}
The function field $K(\cX)$ is called a linearized function field. Many significant families belong to this class of function fields, particularly Artin–Schreier type extensions such as the Hermitian, Suzuki, Artin-Mumford, and Norm-Trace function fields, among others, which have been central object of several works, for instance, see \cite{D1975, G2003, CG2004, GS1991} and the references therein. 

In \cite{N2024}, Navarro provided a basis for the Riemann-Roch spaces $\mathcal{L}(G)$ and a formula for their dimension, where $G$ is a divisor in $K(\mathcal{X})$ with support in the set of totally ramified places of $K(\mathcal{X})/K(x)$. In \cite{ZZ2026}, Zhang and Zhao provided a decomposition of such Riemann-Roch spaces and, as a consequence, they determined the gap set $G(Q)$, where $Q$ is any totally ramified place of degree one in the extension $K(\mathcal{X})/K(x)$. In the same work, by constructing complex functions in $K(\mathcal{X})$, they determined the minimal generating set $\Gamma(\mathbf{Q})$, where $\mathbf{Q}$ is an $n$-tuple of distinct totally ramified places of degree one in $K(\mathcal{X})/K(x)$.

In this article, using the decomposition of Riemann-Roch spaces given by Zhang and Zhao in \cite[Theorem 3.1]{ZZ2026}, we obtain an arithmetic criterion to determine the equality of certain Riemann-Roch spaces, see Theorem \ref{teo_equiv_gap}. With this result, using the concept of discrepancies (see Definition \ref{discrepancy}) and the techniques given in \cite{CMT2026}, we characterize and explicitly determine the sets $\widehat{\Gamma}(\mathbf{Q})$ and $\widehat{\Lambda}(\mathbf{Q})$ of the generalized Weierstrass semigroup $\widehat{H}(\mathbf{Q})$, where $\mathbf{Q}$ is an $n$-tuple of distinct totally ramified places of degree one in $K(\mathcal{X})/K(x)$. These results generalize the principal result presented by Zhang and Zhao in \cite[Theorem 4.1]{ZZ2026}. Furthermore, we apply our results to certain families of algebraic curves.

This paper is organized as follows. In Section \ref{Section 2}, we introduce the preliminaries, notation, and basic results concerning to theory of function fields, generalized Weierstrass semigroups, and linearized function fields. In Section \ref{Section 3}, we provide an arithmetic criterion to determine the absolute and relative maximal elements of $\widehat{H}(\mathbf{Q})$, see Lemmas \ref{lemma_discrepancy_Kummer} and \ref{lemma2_discrepancy_Kummer}, and Theorem \ref{lemma_criterion_Upsilon}. In Section \ref{Section 4}, we provide an explicit description of the sets $\widehat{\Gamma}(\mathbf{Q})$ and $\widehat{\Lambda}(\mathbf{Q})$, see Theorem \ref{teo_widehat_Upsilon}. Finally, we determine the set of maximal elements of the Artin-Mumford curve and a certain curve with many rational places, see Examples \ref{ex1} and \ref{ex2}, respectively.

\section{Preliminaries and Notation} \label{Section 2}

Throughout this article, for $a$ and $b$ integers, we denote by $(a, b)$ the greatest common divisor of $a$ and $b$, and by $b \bmod{a} $ the smallest nonnegative integer congruent with $b$ modulo $a$. For $c\in \R$, we denote by $\floor*{c}$ and $\ceil*{c}$ the floor and ceiling functions of $c$, respectively. We also denote $\N_0 = \N \cup \{0\}$, where $\N$ is the set of positive integers. 

\subsection{Function fields and generalized Weierstrass semigroups} \label{FF GWS}
Let $\fq$ be the finite field with $q$ elements, where $q$ is the power of a prime number $p$.
Let $\cF$ be a function field of one variable over $\fq$ with genus $g(\cF)$. We denote by $\mathcal P_{\cF}$ the set of places in $\cF$, by $\nu_{P}$ the discrete valuation of $\cF$ associated to the place $P\in \cP_{\cF}$, and by $\Div (\cF)$ the group of divisors on $\cF$. For a function $z \in \cF$, $(z)_{\cF}$ and $(z)_\infty$ stand for the principal and pole divisors of the function $z$ in $\cF$, respectively. Given a divisor $G\in \Div(\cF)$, the Riemann-Roch space associated to the divisor $G$ is defined by
$$
\cL(G)=\{z\in \cF: (z)_{\cF}+G\geq 0\}\cup \{0\},
$$  
and we denote by $\ell(G)$ its dimension as vector space over $\fq$.

Let $\Q=(Q_1, \dots, Q_n)$ be an $n$-tuple of distinct rational places in $\cF$. The \emph{Weierstrass semigroup} $H(\Q)$ of $\cF$ at $\Q$ is defined as the set 
$$
H(\mathbf{Q}) := \left\{(a_{1}, \ldots, a_{n}) \in \mathbb{N}_{0}^ {n} :  (z)_{\infty} = \textstyle\sum_{i=1}^ {n} a_{i}Q_{i} \text{ for some }z\in \cF \right\},
$$
and the \emph{generalized Weierstrass semigroup} $\widehat{H}(\Q)$ of $\cF$ at $\Q$ is defined by 
$$
\widehat{H}(\Q):=\{(-\nu_{Q_1}(z),\dots ,-\nu_{Q_n}(z))\in \Z^n : z \in R_\Q\setminus\{0\}\},
$$
where $R_\Q$ denotes the ring of functions on $\cF$ that are regular outside the places in the set
$\{Q_1,\dots, Q_n\}$. By \cite[Proposition 2]{BT2006}, we know that if $q \geq n$, then $H(\Q)=\widehat{H}(\Q)\cap \N_0^n$. 

The elements in the complement set ${G(\mathbf{Q}):=\N^n_0\setminus H(\mathbf{Q})}$ are called \emph{gaps} of $\cF$ at $\mathbf{Q}$ and can be characterized using Riemann-Roch spaces. In fact, for an $n$-tuple $\negalpha=(\al_1, \dots, \al_n)\in  \N_0^n$ we have 
that $\negalpha\in G(\Q)$ if and only if $$\cL(D_{\negalpha}(\Q)) = \cL(D_{\negalpha}(\Q) - Q_i) \text{ for some } 1 \leq i \leq n,$$ where $D_\negalpha(\Q) := \alpha_1 Q_1 + \cdots + \alpha_n Q_n$. On the other hand, a \emph{pure gap} of $\cF$ at $\Q$ is an $n$-tuple $\negalpha = (\alpha_1, \ldots, \alpha_n) \in G(\mathbf{Q})$ such that $$\cL(D_\negalpha(\Q)) = \cL(D_\negalpha(\Q) - Q_i) \text{ for every } 1\leq i \leq n.$$
 The set of pure gaps of $\mathcal{F}$ at $\mathbf{Q}$ will be denoted by $G_0(\mathbf{Q})$. These sets, $G(\mathbf{Q})$ and $G_0(\mathbf{Q})$, have been intensively studied in recent years due to their various applications in other areas, such as coding theory, see e.g. \cite{BQZ2018, CT2005, GKL1993, KNT2020}.
 
The elements of the generalized semigroup $\widehat{H}(\Q)$ also can be characterized using Riemann-Roch spaces as follows. 

\begin{proposition}\cite[Proposition 2.2]{MTT2019} \label{prop_hat_H(Q)}
Let $\negalpha \in \mathbb{Z}^n$ and assume that $q \geq n$. Then
$$\negalpha \in \widehat{H} (\mathbf{Q})\quad \text{if and only if}\quad \cL(D_{\negalpha}(\Q)) \neq \cL(D_{\negalpha}(\Q) - Q_i) \text{ for every } 1\leq i\leq n.$$
\end{proposition}

A fundamental concept in the study of generalized Weierstrass semigroups $\widehat{H}(\Q)$ is notion of maximal elements, which play a crucial role in describing the elements of $\widehat{H}(\Q), G(\Q)$ and $G_0(\Q)$, see \cite{CMT2024, CMT2025, MTT2019, TT2019}. To introduce the concept of maximal elements we need to define the sets below. Set $I:=\{1,\ldots,n\}$. For $i\in I$, a nonempty subset $J\subsetneq I$, and $\negalpha=(\alpha_1,\ldots,\alpha_n)\in \Z^n$, we shall denote
\begin{itemize}
\item [$\bullet$] $\overline{\nabla}_J (\negalpha):=\{\negbeta\in \Z^n : \be_j=\al_j \text{ for }j\in J \text{ and } \be_i < \al_i \text{ for }i\notin J\}$,
\item [$\bullet$] $\nabla_J(\negalpha):=\overline{\nabla}_J(\negalpha)\cap \widehat{H}(\Q)$,
\item [$\bullet$] $\overline{\nabla}(\negalpha):=\cup_{i=1}^{n}\overline{\nabla}_i(\negalpha)$, where $\overline{\nabla}_i(\negalpha):= \overline{\nabla}_{\{i\}}(\negalpha)$, and
\item [$\bullet$] $\nabla(\negalpha):=\overline{\nabla}(\negalpha)\cap \widehat{H}(\Q)$.
\end{itemize}

\begin{definition} \label{defi maximals}
An element $\negalpha\in \widehat{H}(\Q)$ is called maximal if $\nabla(\negalpha)=\emptyset$. If moreover $\nabla_J(\negalpha)=\emptyset$ for every $J\subsetneq I$ with $|J| \geq 2$, we say that $\negalpha$ is absolute maximal. If $\negalpha$ is maximal and $\nabla_J(\negalpha)\neq\emptyset$ for every $J\subsetneq I$ with $|J| \geq 2$, we say that $\negalpha\in \widehat{H}(\Q)$ is relative maximal. The sets of absolute and relative maximal elements in $\widehat{H}(\Q)$ will be denoted, respectively, by $\widehat{\Gamma}(\Q)$ and $\widehat{\Lambda}(\Q)$.
\end{definition}

Observe that the notions of absolute and relative maximality coincide when $n=2$. Now, define the following sets $$\Gamma(\Q):=\widehat{\Gamma}(\Q)\cap \N^n \quad \text{and}\quad  \widehat{\Lambda}(\Q):=\Lambda(\Q)\cap \N^n.$$ 
The set $\Gamma(\mathbf{Q})$, called the \emph{minimal generating set} of the Weierstrass semigroup $H(\mathbf{Q})$, completely determines the Weierstrass semigroup $H(\mathbf{Q})$, see \cite[Theorem 7]{MG2004}. On the other hand, the sets $\widehat{\Gamma}(\mathbf{Q})$ and $\Lambda(\mathbf{Q})$ completely determine the generalized Weierstrass semigroup $\widehat{H}(\mathbf{Q})$ and the set of pure gaps $G_0(\mathbf{Q})$, respectively. In fact, for $\negbeta^1,\dots, \negbeta^s\in \Z^n$, the least upper bound and the greatest lower bound  of $\negbeta^1,\dots, \negbeta^s$ are defined as
$$
\lub(\negbeta^1,\dots , \negbeta^s):=(\max\{\be_1^1, \dots, \be_1^s\},\dots , \max\{\be_n^1, \dots, \be_n^s\})\in \Z^n
$$
and 
$$
\glb(\negbeta^1,\dots , \negbeta^s):=(\min\{\be_1^1, \dots, \be_1^s\},\dots , \min\{\be_n^1, \dots, \be_n^s\})\in \Z^n,
$$
respectively. With these concepts, we have the following characterizations of $\widehat{H}(\Q)$ and $G_0(\Q)$.

\begin{theorem}\cite[Theorem 3.4]{MTT2019}
Assume $2 \leq n\leq q$. Then
$$
\widehat{H}(\Q)=\{\lub(\negbeta^1, \dots, \negbeta^n): \negbeta^1, \dots, \negbeta^n\in \widehat{\Gamma}(\Q)\}.
$$
\end{theorem}

\begin{theorem}\cite[Corollary 4.2]{CMT2025}\label{coro_G0(Q)}
	Assume $2\leq n\leq q$. Then
	\begin{equation*}
		G_0(\Q)=\left\{\glb(\negbeta^1, \dots, \negbeta^n): \negbeta^l\in \Lambda(\Q)\text{ and } \beta_l^l<\min\{\beta_l^1, \dots, \widehat{\beta_l^l}, \dots, \beta_l^n\}\text{ for }1\leq l\leq n\right\}.
	\end{equation*}
\end{theorem}

Next, we present results that characterize the absolute and relative maximal elements of $\widehat{H}(\Q)$. To do this, we introduce the concept of \emph{discrepancy}.

\begin{definition} \label{discrepancy}
Let $Q_1$ and $Q_2$ be distinct rational places in $\cF$. A divisor $A\in \Div(\cF)$ is called a discrepancy with respect to $Q_1$ and $Q_2$ if
$$
\cL(A)\neq \cL(A-Q_1)\quad \text{and}\quad\cL(A-Q_2)=\cL(A-Q_1-Q_2).
$$ 
\end{definition}

The following results characterize the elements of the sets $\widehat{\Gamma}(\mathbf{Q})$ and $\widehat{\Lambda}(\mathbf{Q})$ using the notion of discrepancy.

\begin{proposition}\cite[Proposition 3]{TT2019FF} \label{prop_discrepancia_Gamma}
Let $\negalpha \in \mathbb{Z}^n$ and assume that $q\geq n$. The following statements are equivalent:
\begin{enumerate}[\rm (i)]
\item $\negalpha \in \widehat{\Gamma}(\mathbf{Q})$.
\item $D_{\negalpha}(\mathbf{Q})$ is a discrepancy with respect to any pair of distinct places in $\{Q_1,\ldots,Q_{n}\}$.
\end{enumerate}
\end{proposition}

\begin{proposition}\cite[Proposition 2.8]{TT2019} \label{prop_discrepamcia_Lambda}
Let $\negalpha\in \mathbb{Z}^n$ and assume that $q \geq n$. The following statements are equivalent:
\begin{enumerate}[\rm (i)]
\item $\negalpha\in \widehat{\Lambda}(\Q)$.

\item $D_{\negalpha-\textbf{1}}(\Q)+Q_i+Q_j$ is a discrepancy with respect to $Q_i$ and $Q_j$ for every $1\leq i,j\leq n$ with $i\neq j$, where $\neg1$ is the $n$-tuple whose coordinates are all $1$.
\end{enumerate}
\end{proposition}

Another tool from the theory of function fields that we will use in this work is the restriction of divisors. For a finite extension of function fields $\cG \subseteq \cF$ and $D \in \Div(\cF)$, we can write
$$
D = \sum_{P \in \mathcal{P}_{\cG}} \sum_{Q \in \mathcal{P}_{\cF} , Q|P} n_Q Q,
$$
where $Q|P$ means that $P$ lies under $Q$. The restriction of the divisor $D$ to function field $\cG$, denoted by $D\rvert_{\cG}$, is defined by
$$
D\rvert_{\cG}:=\sum_{P\in \cP_{\cG}}\min\left\{ \left\lfloor \dfrac{n_Q}{e(Q|P)}\right\rfloor : Q|P\right\}P,
$$ 
where $e(Q|P)$ is the ramification index of $Q$ over $P$.

\subsection{Linearized function fields}\label{subsection 2.2}
Let $s, m$ be positive integers, $q$ be a power of a prime $p$, and $K=\fqs$.  
Consider the linearized polynomial  
$$
L(y) = \sum_{k=0}^{m} a_k y^{q^{k}} \in K[y],
$$  
where $m\leq s$, $a_m \neq 0$, $a_0 \neq 0$, and assume all its roots lie in $K$. Let $h(x)\in K(x)$ be a rational function such that its principal divisor in $K(x)$ is given by
$$(h(x))_{K(x)}=\gamma_1T_1+\gamma_2T_2+\cdots+\gamma_uT_u-\la_1P_1-\la_2P_2-\cdots-\la_rP_r,$$
where 
\begin{itemize}
	\item $T_1, \dots, T_u$ are the places in $K(x)$ associated to the zeros of $h(x)$, 
	\item $P_1, \dots, P_r$ are the places in $K(x)$ associated to the poles of $h(x)$, and 
	\item $\gamma_1, \dots, \gamma_u, \la_1, \dots, \la_r$ are positive integers such that $p\nmid\la_k$ for every $1\leq k \leq r$. 
\end{itemize}
Now, consider the algebraic curve over $K$ defined by the affine equation
\begin{equation}\label{Xequation}
\cX: \quad L(y)=h(x).
\end{equation}
The function field $K(\mathcal{X})$ is called a linearized function field and we have the property $[K(\mathcal{X}):K(x)]=q^m$. 
From \cite[Proposition 3.7.10]{S2009} and \cite[Proposition 3.1]{N2024}, we have that $P_1, \dots, P_r$ are all the places totally ramified in $K(\cX)/K(x)$ and $T_1, \dots, T_u$ split completely in $K(\cX)/K(x)$. 
For each $1\leq i \leq r$, let $Q_i$ be the only place in $K(\cX)$ lying over $P_i$.

Furthermore, also from \cite[Proposition 3.1]{N2024}, we know that the principal divisor associated to $y$ in $K(\cX)$ is given by 
\begin{equation}\label{divisor_y}
(y)_{K(\cX)}=\gamma_1R_1+\gamma_2R_2+\dots +\gamma_uR_u-\la_1Q_1-\la_2Q_2-\cdots -\la_rQ_r,
\end{equation}
where $R_i$ is the place associated to the only zero of $y$ in $K(\cX)$ lying over $T_i$.

On the other hand, in \cite{ZZ2026} the authors provided a theorem that allows one to obtain a decomposition of certain Riemann-Roch spaces in this class of function fields.

\begin{theorem}\cite[Theorem 3.1]{ZZ2026}\label{teo_decomposition}
	Let $\cX$ be the curve as in (\ref{Xequation}) and $G\in \Div(\cX)$ be a divisor with support contained in the set of totally ramified places $\{Q_1, Q_2, \dots, Q_r\}$ of $K(\cX)/K(x)$. Then
	$$
	\cL(G)=\displaystyle\bigoplus_{i=0}^{q^m-1}\cL((G+i(y))\rvert_{K(x)})y^i.
	$$
\end{theorem}

\section{On maximal elements in generalized Weierstrass semigroups}\label{Section 3}

Using the same notation introduced in Subsection \ref{subsection 2.2}, let $\mathcal{X}$ be the curve given in (\ref{Xequation}) and, for $1\leq n \leq r$, let $\mathbf{Q}=(Q_1, Q_2, \dots, Q_n)$ be an $n$-tuple of distinct totally ramified places (not necessarily rational) in the extension $K(\mathcal{X})/K(x)$. The objective of this section is to provide an arithmetic criterion that completely characterizes the absolute and relative maximal elements of the generalized Weierstrass semigroup $\widehat{H}(\mathbf{Q})$ when $\Q$ is an $n$-tuple of distinct totally ramified places of degree one.

First of all, we establish an arithmetic criterion to determine when an $n$-tuple of integers $\boldsymbol{\alpha}=(\alpha_1, \dots, \alpha_n)$ satisfies the equality of Riemann-Roch spaces $$\mathcal{L}(D_{\boldsymbol{\alpha}}(\mathbf{Q}))=\mathcal{L}(D_{\boldsymbol{\alpha}}(\mathbf{Q})-Q_j)\quad  \text{for some }1\leq j\leq n,$$ where $D_{\boldsymbol{\alpha}}(\mathbf{Q}):=\sum_{k=1}^{n}\alpha_kQ_k$. To this end, we will use the decomposition of Riemann-Roch spaces given in Theorem \ref{teo_decomposition} and the following lemma.

\begin{lemma}\cite[Lemma 3.3]{ABQ2019}\label{lemma_floor}
Let $\al, \la\in \Z$ and $m\in \N$ be such that $(\la, m)=1$. Then there exists a unique integer $t\in \{0, \dots, m-1\}$ such that $\al+t\la\equiv 0 \bmod{m}$ and, for $i\in \Z$, we have
$$
\floor*{\frac{\al+i\la}{m}}=\left\{
\begin{array}{ll}
\displaystyle\floor*{\frac{\al-1+i\la}{m}}+1=\ceil*{\frac{\al}{m}}+\floor*{\frac{i\la}{m}}, & \text{if } i\equiv t \bmod{m},\\
\\
\displaystyle\floor*{\frac{\al-1+i\la}{m}}, & \text{if } i\not\equiv t \bmod{m}.
\end{array}\right.
$$  
\end{lemma}

\begin{theorem}\label{teo_equiv_gap}
Let $\Q=(Q_1, \dots, Q_n)$ an $n$-tuple of distinct totally ramified places in the extension $K(\cX)/K(x)$, $\negalpha=(\al_1, \dots, \al_n) \in \Z^n$, and $1\leq j\leq n$. Then $$\cL(D_{\negalpha}(\Q))=\cL(D_{\negalpha}(\Q)-Q_j)$$ if and only if
$$
\sum_{k=1}^{n}\floor*{\frac{\al_k-t_j\la_k}{q^m}}\deg(P_k)<\sum_{k=n+1}^{r}\ceil*{\frac{t_j\la_k}{q^m}}\deg(P_k),
$$
where $t_j\in\{0, \dots, q^m-1\}$ is the unique integer such that $\al_j-t_j\la_j\equiv 0 \bmod{q^m}$.
\end{theorem}
\begin{proof}
From (\ref{divisor_y}) we have that the principal divisor of $y$ in $K(\cX)$ is given by
\begin{equation*}
	(y)_{K(\cX)}=\gamma_1R_1+\gamma_2R_2+\dots +\gamma_uR_u-\la_1Q_1-\la_2Q_2-\cdots -\la_rQ_r.
\end{equation*}
Thus, for each $0\leq i \leq q^m-1$, we have
$$
D_\negalpha(\Q)+i(y)=\sum_{k=1}^{u}i\gamma_kR_k+\sum_{k=1}^{n}(\al_k-i\la_k)Q_k-\sum_{k=n+1}^{r}i\la_kQ_k
$$
and therefore
$$
[D_\negalpha(\Q)+i(y)]\rvert_{K(x)}=\sum_{k=1}^{n}\floor*{\frac{\al_k-i\la_k}{q^m}}P_k-\sum_{k=n+1}^{r}\ceil*{\frac{i\la_k}{q^m}}P_k.
$$
Analogously,
$$
[D_\negalpha(\Q)-Q_j+i(y)]\rvert_{K(x)}=\floor*{\frac{\al_j-1-i\la_j}{q^m}}P_j+\sum_{\substack{k=1\\ k\neq j}}^{n}\floor*{\frac{\al_k-i\la_k}{q^m}}P_k-\sum_{k=n+1}^{r}\ceil*{\frac{i\la_k}{q^m}}P_k.
$$
From Theorem \ref{teo_decomposition}, $\cL(D_{\negalpha}(\Q))=\cL(D_{\negalpha}(\Q)-Q_j)$ if and only if
$$
\cL \left(\floor*{\frac{\al_j-i\la_j}{q^m}}P_j+D\right)=\cL \left(\floor*{\frac{\al_j-1-i\la_j}{q^m}}P_j+D\right)$$
for every $0\leq i \leq q^m-1$, where 
$$D=\sum_{\substack{k=1\\ k\neq j}}^{n}\floor*{\frac{\al_k-i\la_k}{q^m}}P_k-\sum_{k=n+1}^{r}\ceil*{\frac{i\la_k}{q^m}}P_k.$$ 

Since for divisors $A, B\in \Div (K(x))$ such that $A\leq B$, we have $\cL(A)=\cL(B)$ if and only if $\deg (B)< 0$ or $A=B$, we conclude that $\cL(D_{\negalpha}(\Q))=\cL(D_{\negalpha}(\Q)-Q_j)$ if and only if, for each $0\leq i \leq q^m-1$, 
$$
\sum_{k=1}^{n}\floor*{\frac{\al_k-i\la_k}{q^m}}\deg(P_k)-\sum_{k=n+1}^{r}\ceil*{\frac{i\la_k}{q^m}}\deg(P_k) <0\quad \text{or}\quad
\floor*{\frac{\al_j-i\la_j}{q^m}}=\floor*{\frac{\al_j-1-i\la_j}{q^m}}.
$$

From Lemma \ref{lemma_floor}, there exists a unique integer $t_j\in \{0, \dots, q^m-1\}$ such that $\al_j-t_j\la_j\equiv 0 \bmod{q^m}$. Finally, also from Lemma \ref{lemma_floor}, we have that $\cL(D_{\negalpha}(\Q))=\cL(D_{\negalpha}(\Q)-Q_j)$ if and only if 
$$
\sum_{k=1}^{n}\floor*{\frac{\al_k-t_j\la_k}{q^m}}\deg(P_k)-\sum_{k=n+1}^{r}\ceil*{\frac{t_j\la_k}{q^m}}\deg(P_k) <0.
$$
\end{proof}

Note that the previous theorem generalizes the criterion to obtain gaps at one totally ramified place of $K(\mathcal{X})/K(x)$ given in \cite[Lemma 3.3]{ZZ2026}. In fact, as an immediate consequence of Theorem \ref{teo_equiv_gap} and Lemma \ref{lemma_floor}, we obtain the following result, which coincides with \cite[Lemma 3.3]{ZZ2026}.
\begin{corollary}\label{coro_criterion_gap_onepoint}
Let $1\leq l \leq r$ be such that $Q_l$ is a rational place. Then, for $\al\in \N$, we have that $\al\in G(Q_l)$ if and only if 
$$
\ceil*{\frac{\al}{q^m}}<\sum_{k=1}^{r}\ceil*{\frac{t\la_k}{q^m}}\deg(P_k),
$$
where $t\in\{0, \dots, q^m-1\}$ is the unique integer such that $\al-t\la_l\equiv 0 \bmod{q^m}$.
\end{corollary}
\begin{proof}
Recall that $\al \in G(Q_l)$ if and only if $\cL(\al Q_l) = \cL((\al-1)Q_l)$. Thus, by Theorem \ref{teo_equiv_gap}, $\al \in G(Q_l)$ if and only if
\begin{equation*}
	\floor*{\frac{\al-t\la_l}{q^m}} < \sum_{\substack{k=1 \\ k \neq l}}^{r} \ceil*{\frac{t\la_k}{q^m}} \deg(P_k),
\end{equation*}
where $t \in \{0, \dots, q^m-1\}$ is the unique integer satisfying $\al - t\la_l \equiv 0 \bmod{q^m}$. The result follows since $\lfloor\frac{\al-t\la_l}{q^m}\rfloor=\lceil\frac{\al}{q^m}\rceil-\lceil\frac{t\la_l}{q^m}\rceil$ by Lemma \ref{lemma_floor}.
\end{proof}

 In the remainder of this work, we will assume that $\mathbf{Q}=(Q_1, \dots, Q_n)$ is an $n$-tuple of distinct totally ramified places of degree one in $K(\mathcal{X})/K(x)$.

Now, using Theorem \ref{teo_equiv_gap} and the characterization of the absolute maximal elements of $\widehat{H}(\Q)$ via discrepancies given in Proposition \ref{prop_discrepancia_Gamma}, we provide an arithmetic criterion to completely determine the elements of the set $\widehat{\Gamma}(\mathbf{Q})$.

\begin{lemma}\label{lemma_discrepancy_Kummer}
Assume that $2\leq n \leq q^s$ and let $\negalpha = (\alpha_1, \ldots, \alpha_n) \in \Z^n$. Then $\negalpha \in \widehat{\Gamma}(\Q)$ if and only if there exists a unique $t\in \{0, \dots, q^m-1\}$ such that $\al_k-t\la_k\equiv 0 \bmod{q^m}$ for every $1\leq k \leq n$ and
$$
\sum_{k=1}^{n}\ceil*{\frac{\al_k}{q^m}}-\sum_{k=1}^{r}\ceil*{\frac{t\la_k}{q^m}}\deg(P_k)=0.
$$
\end{lemma}
\begin{proof}
First note that, from Proposition \ref{prop_discrepancia_Gamma}, $\negalpha \in \widehat{\Gamma}(\Q)$ if and only if $D_\negalpha(\Q)$ is a discrepancy with respect to any pair of distinct places in $\{Q_1, Q_2, \dots, Q_n\}$. Thus, from Definition \ref{discrepancy} and Proposition \ref{prop_hat_H(Q)}, $\negalpha \in \widehat{\Gamma}(\Q)$ if and only if 
$$\negalpha \in \widehat{H}(\Q) \text{ and }\cL(D_{\negalpha-\nege_j}(\Q))=\cL(D_{\negalpha-\nege_j}(\Q)-Q_i) \text{ for every }1\leq i, j \leq n \text{ with }i \neq j,
$$ where $\nege_j$ is the $n$-tuple whose the $j$-th coordinate is $1$ and the others are $0$. Furthermore, from Proposition \ref{prop_hat_H(Q)} and Theorem \ref{teo_equiv_gap}, we have that

$\bullet$ $\negalpha \in \widehat{H}(\Q)$ if and only if
\begin{equation}\label{equation_1}
\sum_{k=1}^{n}\floor*{\frac{\al_k-t_i\la_k}{q^m}}\geq \sum_{k=n+1}^{r}\ceil*{\frac{t_i\la_k}{q^m}}\deg(P_k)\quad \text{for every } 1\leq i \leq n,
\end{equation} 
where $t_i\in\{0, \dots, q^m-1\}$ is the unique integer such that $\al_i-t_i\la_i\equiv 0 \bmod{q^m}$, and 

$\bullet$ $\cL(D_{\negalpha-\nege_j}(\Q))=\cL(D_{\negalpha-\nege_j}(\Q)-Q_i)$ for every $1\leq i, j\leq n$ with $i \neq j$ if and only if
\begin{equation}\label{equation_2}
\sum_{\substack{k=1\\k\neq j}}^{n}\floor*{\frac{\al_k-t_i\la_k}{q^m}}+\floor*{\frac{\al_j-1-t_i\la_j}{q^m}}< \sum_{k=n+1}^{r}\ceil*{\frac{t_i\la_k}{q^m}}\deg(P_k) \text{ for every } i, j\in I \text{ with } i\neq j. 
\end{equation}

Now, suppose that $\negalpha \in \widehat{\Gamma}(\Q)$. So, Equations (\ref{equation_1}) and  (\ref{equation_2}) are satisfied. Thus, for every $1\leq i, j\leq n$ with $i\neq j$,
\begin{align*}
0&\leq \sum_{k=1}^{n}\floor*{\frac{\al_k-t_i\la_k}{q^m}}- \sum_{k=n+1}^{r}\ceil*{\frac{t_i\la_k}{q^m}}\deg(P_k) & \text{(from Equation (\ref{equation_1}))}\\
&=\sum_{\substack{k=1\\k\neq j}}^{n}\floor*{\frac{\al_k-t_i\la_k}{q^m}}+\floor*{\frac{\al_j-1-t_i\la_j}{q^m}}-\sum_{k=n+1}^{r}\ceil*{\frac{t_i\la_k}{q^m}}\deg(P_k)\\
&\quad +\floor*{\frac{\al_j-t_i\la_j}{q^m}} -\floor*{\frac{\al_j-1-t_i\la_j}{q^m}}\\
&\leq -1+\floor*{\frac{\al_j-t_i\la_j}{q^m}}-\floor*{\frac{\al_j-1-t_i\la_j}{q^m}}& \text{(from Equation (\ref{equation_2}))}\\
&\leq -1+\floor*{\frac{\al_j-t_i\la_j-(\al_j-1-t_i\la_j)}{q^m}}+1\\
&=0.
\end{align*}

Thus, we conclude that
\begin{equation}\label{equation_3}
\sum_{k=1}^{n}\floor*{\frac{\al_k-t_i\la_k}{q^m}}-\sum_{k=n+1}^{r}\ceil*{\frac{t_i\la_k}{q^m}}\deg(P_k)=0 \quad \text{for every }1\leq i\leq n  
\end{equation}
and
\begin{equation}\label{equation_4}
\floor*{\frac{\al_j-t_i\la_j}{q^m}}=\floor*{\frac{\al_j-1-t_i\la_j}{q^m}}+1\quad  \text{for every } 1\leq i, j\leq n \text{ with } i\neq j.
\end{equation}
From Lemma \ref{lemma_floor} and Equation (\ref{equation_4}) we obtain $t:=t_1=t_2=\dots=t_n$. Thus, $\al_k-t\la_k\equiv 0 \bmod{q^m}$ for every $1\leq k \leq n$. Furthermore, from Lemma \ref{lemma_floor} and Equation (\ref{equation_3}) we conclude that
\begin{align*}
\sum_{k=1}^{n}\ceil*{\frac{\al_k}{q^m}}-\sum_{k=1}^{r}\ceil*{\frac{t\la_k}{q^m}}\deg(P_k)&=\sum_{k=1}^{n}\left(\floor*{\frac{\al_k-t\la_k}{q^m}}-\floor*{\frac{-t\la_k}{q^m}}\right)-\sum_{k=1}^{r}\ceil*{\frac{t\la_k}{q^m}}\deg(P_k)\\
&=\sum_{k=1}^{n}\left(\floor*{\frac{\al_k-t\la_k}{q^m}}+\ceil*{\frac{t\la_k}{q^m}}\right)-\sum_{k=1}^{r}\ceil*{\frac{t\la_k}{q^m}}\deg(P_k)\\
&=\sum_{k=1}^{n}\floor*{\frac{\al_k-t\la_k}{q^m}}-\sum_{k=n+1}^{r}\ceil*{\frac{t\la_k}{q^m}}\deg(P_k)\\
&=0.
\end{align*}

Conversely, suppose that there exists a unique $t\in \{0, \dots, q^m-1\}$ such that $\al_k-t\la_k\equiv 0 \bmod{q^m}$ for every $1\leq k \leq n$ and
$$
\sum_{k=1}^{n}\ceil*{\frac{\al_k}{q^m}}-\sum_{k=1}^{r}\ceil*{\frac{t\la_k}{q^m}}\deg(P_k)=0.
$$
From Lemma \ref{lemma_floor} we obtain 
$$
\sum_{k=1}^{n}\floor*{\frac{\al_k-t\la_k}{q^m}}- \sum_{k=n+1}^{r}\ceil*{\frac{t\la_k}{q^m}}\deg(P_k)=\sum_{k=1}^{n}\ceil*{\frac{\al_k}{q^m}}-\sum_{k=1}^{r}\floor*{\frac{t\la_k}{q^m}}\deg(P_k)=0
$$ 
and, for every $1\leq j\leq n$,
\begin{align*}
&\sum_{\substack{k=1\\k\neq j}}^{n}\floor*{\frac{\al_k-t\la_k}{q^m}}+\floor*{\frac{\al_j-1-t\la_j}{q^m}}- \sum_{k=n+1}^{r}\ceil*{\frac{t\la_k}{q^m}}\deg(P_k)\\
&=\sum_{\substack{k=1\\k\neq j}}^{n}\floor*{\frac{\al_k-t\la_k}{q^m}}+\floor*{\frac{\al_j-t\la_j}{q^m}}-1-\sum_{k=n+1}^{r}\ceil*{\frac{t\la_k}{q^m}}\deg(P_k)\\
&=\sum_{k=1}^{n}\floor*{\frac{\al_k-t\la_k}{q^m}}-1-\sum_{k=n+1}^{r}\ceil*{\frac{t\la_k}{q^m}}\deg(P_k)\\
&=-1.
\end{align*} 
Thus, $\negalpha$ satisfies Equations (\ref{equation_1}) and (\ref{equation_2}), which implies that $\negalpha \in \widehat{\Gamma}(\mathbf{Q})$, completing the proof.
\end{proof}

Now, using Theorem \ref{teo_equiv_gap} and the characterization of the relative maximal elements of $\widehat{H}(\Q)$ via discrepancies given in Proposition \ref{prop_discrepamcia_Lambda}, we provide an arithmetic criterion to completely determine the elements of the set $\widehat{\Lambda}(\mathbf{Q})$.

\begin{lemma}\label{lemma2_discrepancy_Kummer}
Assume that $2 \leq n \leq q^s$ and let $\negalpha=(\al_1, \dots, \al_n) \in \Z^n$. Then $\negalpha\in \widehat{\Lambda}(\Q)$ if and only if there exists a unique $t\in \{0, \dots, q^m-1\}$ such that $\al_k-t\la_k\equiv 0 \bmod{q^m}$ for every $1\leq k \leq n$ and
$$
\sum_{k=1}^{n}\ceil*{\frac{\al_k}{q^m}}-\sum_{k=1}^{r}\ceil*{\frac{t\la_k}{q^m}}\deg(P_k)=n-2.
$$
\end{lemma}
\begin{proof}
From Proposition \ref{prop_discrepamcia_Lambda}, $\negalpha\in \widehat{\Lambda}(\Q)$ if and only if  $D_{\negalpha-\textbf{1}}(\Q)+Q_i+Q_j$ is a discrepancy with respect to $Q_i$ and $Q_j$ for every $1\leq i,j\leq n$ with $i\neq j$. Thus, from Definition \ref{discrepancy}, $\negalpha\in \widehat{\Lambda}(\Q)$ if and only if $$\cL(D_{\negalpha-\neg1}(\Q)+Q_i+Q_j)\neq \cL(D_{\negalpha-\neg1}(\Q)+Q_j) \text{ for every }1\leq i, j \leq n\text{ with }i \neq j
$$ and $$
\cL(D_{\negalpha-\neg1}(\Q)+Q_i)=\cL(D_{\negalpha-\neg1}(\Q))\text{ for every }1\leq i\leq n.
$$ 
Furthermore, from Theorem \ref{teo_equiv_gap}, we have that

$\bullet$ $\cL(D_{\negalpha-\neg1}(\Q)+Q_i+Q_j)\neq \cL(D_{\negalpha-\neg1}(\Q)+Q_j)$ for every $1\leq i, j \leq n$ with $i \neq j$ if and only if
\begin{equation}\label{equation_5}
\sum_{\substack{k=1\\k\neq i, j}}^{n}\floor*{\frac{\al_k-1-t_i\la_k}{q^m}}+\floor*{\frac{\al_i-t_i\la_i}{q^m}}+\floor*{\frac{\al_j-t_i\la_j}{q^m}}\geq\sum_{k=n+1}^{r}\ceil*{\frac{t_i\la_k}{q^m}}\deg(P_k)
\end{equation}
for every $1\leq i, j\leq n$ with $i\neq j$, where $t_i\in \{0, \dots, q^m-1\}$ is the unique integer such that $\al_i-t_i\la_i\equiv 0\bmod{q^m}$, and

$\bullet$ $\cL(D_{\negalpha-\neg1}(\Q)+Q_i)=\cL(D_{\negalpha-\neg1}(\Q))$ for every $1\leq i\leq n$ if and only if
\begin{equation}\label{equation_6}
\sum_{\substack{k=1\\k\neq i}}^{n}\floor*{\frac{\al_k-1-t_i\la_k}{q^m}}+\floor*{\frac{\al_i-t_i\la_i}{q^m}}<\sum_{k=n+1}^{r}\ceil*{\frac{t_i\la_k}{q^m}}\deg(P_k)\quad  \text{for every } 1\leq i\leq n. 
\end{equation}

Suppose that $\negalpha\in \widehat{\Lambda}(\Q)$. So, Equations (\ref{equation_5}) and (\ref{equation_6}) are satisfied. Thus, for every $1\leq i, j\leq n$ with $i\neq j$,
\begin{align*}
0&\leq \sum_{\substack{k=1\\k\neq i, j}}^{n}\floor*{\frac{\al_k-1-t_i\la_k}{q^m}}+\floor*{\frac{\al_i-t_i\la_i}{q^m}}+\floor*{\frac{\al_j-t_i\la_j}{q^m}}\\
&\quad -\sum_{k=n+1}^{r}\ceil*{\frac{t_i\la_k}{q^m}}\deg(P_k)& \text{(from Equation (\ref{equation_5}))}\\
&=\sum_{\substack{k=1\\k\neq i}}^{n}\floor*{\frac{\al_k-1-t_i\la_k}{q^m}}+\floor*{\frac{\al_i-t_i\la_i}{q^m}}-\sum_{k=n+1}^{r}\ceil*{\frac{t_i\la_k}{q^m}}\deg(P_k)\\
&\quad+\floor*{\frac{\al_j-t_i\la_j}{q^m}}-\floor*{\frac{\al_j-1-t_i\la_j}{q^m}}\\
&\leq -1 +\floor*{\frac{\al_j-t_i\la_j}{q^m}}-\floor*{\frac{\al_j-1-t_i\la_j}{q^m}}& \text{(from Equation (\ref{equation_6}))}\\
&\leq -1+\floor*{\frac{\al_j-t_i\la_j-(\al_j-1-t_i\la_j)}{q^m}}+1\\
&=0.
\end{align*}

This implies 
\begin{equation}\label{equation_7}
\sum_{\substack{k=1\\k\neq i, j}}^{n}\floor*{\frac{\al_k-1-t_i\la_k}{q^m}}+\floor*{\frac{\al_i-t_i\la_i}{q^m}}+\floor*{\frac{\al_j-t_i\la_j}{q^m}}-\sum_{k=n+1}^{r}\ceil*{\frac{t_i\la_k}{m}}\deg(P_k)=0
\end{equation}
and
\begin{equation}\label{equation_8}
\floor*{\frac{\al_j-t_i\la_j}{q^m}}=\floor*{\frac{\al_j-1-t_i\la_j}{q^m}}+1
\end{equation}
for every $1\leq i, j\leq n$ with $i\neq j$. From Lemma \ref{lemma_floor} and Equation (\ref{equation_8}) we conclude that $t:=t_1=t_2=\dots=t_n$ and thus $\al_k-t\la_k\equiv 0 \bmod{q^m}$ for every $1\leq k \leq n$. Moreover, from Lemma \ref{lemma_floor} and Equation (\ref{equation_7}), we obtain
\begin{align*}
&\sum_{k=1}^{n}\ceil*{\frac{\al_k}{q^m}}-\sum_{k=1}^{r}\ceil*{\frac{t\la_k}{q^m}}\deg(P_k)\\
&=\sum_{k=1}^{n}\left(\floor*{\frac{\al_k-t\la_k}{q^m}}-\floor*{\frac{-t\la_k}{q^m}}\right)-\sum_{k=1}^{r}\ceil*{\frac{t\la_k}{q^m}}\deg(P_k)\\
&=\sum_{k=1}^{n}\left(\floor*{\frac{\al_k-t\la_k}{q^m}}+\ceil*{\frac{t\la_k}{q^m}}\right)-\sum_{k=1}^{r}\ceil*{\frac{t\la_k}{q^m}}\deg(P_k)\\
&=\sum_{k=1}^{n}\floor*{\frac{\al_k-t\la_k}{q^m}}-\sum_{k=n+1}^{r}\ceil*{\frac{t\la_k}{q^m}}\deg(P_k)\\
&=\sum_{\substack{k=1\\k\neq i, j}}^{n}\floor*{\frac{\al_k-t\la_k}{q^m}}+\floor*{\frac{\al_i-t\la_i}{q^m}}+\floor*{\frac{\al_j-t\la_j}{q^m}}-\sum_{k=n+1}^{r}\ceil*{\frac{t\la_k}{q^m}}\deg(P_k)\\
&=\sum_{\substack{k=1\\k\neq i, j}}^{n}\left(\floor*{\frac{\al_k-1-t\la_k}{q^m}}+1\right)+\floor*{\frac{\al_i-t\la_i}{q^m}}+\floor*{\frac{\al_j-t\la_j}{q^m}}-\sum_{k=n+1}^{r}\ceil*{\frac{t\la_k}{q^m}}\deg(P_k)\\
&=n-2+\sum_{\substack{k=1\\k\neq i, j}}^{n}\floor*{\frac{\al_k-1-t\la_k}{q^m}}+\floor*{\frac{\al_i-t\la_i}{q^m}}+\floor*{\frac{\al_j-t\la_j}{q^m}}-\sum_{k=n+1}^{r}\ceil*{\frac{t\la_k}{q^m}}\deg(P_k)\\
&=n-2.
\end{align*}

Conversely, suppose that there exists a unique $t\in \{0, \dots, q^m-1\}$ such that $\al_k-t\la_k\equiv 0 \bmod{q^m}$ for every $1\leq k \leq n$ and
$$
\sum_{k=1}^{n}\ceil*{\frac{\al_k}{q^m}}-\sum_{k=1}^{r}\ceil*{\frac{t\la_k}{q^m}}\deg(P_k)=n-2.
$$
From Lemma \ref{lemma_floor}, for every $1\leq i, j\leq n$ with $i\neq j$ we have that 
\begin{align*}
&\sum_{\substack{k=1\\k\neq i, j}}^{n}\floor*{\frac{\al_k-1-t\la_k}{q^m}}+\floor*{\frac{\al_i-t\la_i}{q^m}}+\floor*{\frac{\al_j-t\la_j}{q^m}}-\sum_{k=n+1}^{r}\ceil*{\frac{t\la_k}{m}}\deg(P_k)\\
&=\sum_{\substack{k=1\\k\neq i, j}}^{n}\left(\ceil*{\frac{\al_k}{q^m}}-\ceil*{\frac{t\la_k}{q^m}}-1\right)+\ceil*{\frac{\al_i}{q^m}}-\ceil*{\frac{t\la_i}{q^m}}+\ceil*{\frac{\al_j}{q^m}}-\ceil*{\frac{t\la_j}{q^m}}-\sum_{k=n+1}^{r}\ceil*{\frac{t\la_k}{m}}\deg(P_k)\\
&=-n+2+\sum_{k=1}^{n}\ceil*{\frac{\al_k}{q^m}}-\sum_{k=1}^{r}\ceil*{\frac{t\la_k}{q^m}}\deg(P_k)\\
&=0
\end{align*}
and, for every $1\leq i\leq n$,
\begin{align*}
&\sum_{\substack{k=1\\k\neq i}}^{n}\floor*{\frac{\al_k-1-t\la_k}{q^m}}+\floor*{\frac{\al_i-t\la_i}{q^m}}-\sum_{k=n+1}^{r}\ceil*{\frac{t\la_k}{m}}\deg(P_k)\\
&=\sum_{\substack{k=1\\k\neq i}}^{n}\left(\ceil*{\frac{\al_k}{q^m}}-\ceil*{\frac{t\la_k}{q^m}}-1\right)+\ceil*{\frac{\al_i}{q^m}}-\ceil*{\frac{t\la_i}{q^m}}-\sum_{k=n+1}^{r}\ceil*{\frac{t\la_k}{m}}\deg(P_k)\\
&=-n+1+\sum_{k=1}^{n}\ceil*{\frac{\al_k}{q^m}}-\sum_{k=1}^{r}\ceil*{\frac{t\la_k}{m}}\deg(P_k)\\
&=-1.
\end{align*} 
Thus, $\negalpha$ satisfies Equations (\ref{equation_5}) and (\ref{equation_6}), which implies that $\negalpha \in \widehat{\Lambda}(\mathbf{Q})$, completing the proof.
\end{proof}

Combining the previous lemmas, we obtain the following arithmetic characterization of the absolute and relative maximal elements of the generalized Weierstrass semigroup $\widehat{H}(\mathbf{Q})$. For abbreviation, in the remainder of this work, $\Upsilon$ will denote either $\Gamma$ or $\Lambda$.
\begin{theorem}\label{lemma_criterion_Upsilon}
Assume that $2 \leq n \leq q^s$ and let $\negalpha=(\al_1, \dots, \al_n)\in \Z^n$. Then $\negalpha\in \widehat{\Upsilon}(\Q)$ if and only if there exists a unique $t\in \{0, \dots, q^m-1\}$ such that $\al_k-t\la_k\equiv 0 \bmod{q^m}$ for every $1\leq k \leq n$ and 
$$
\sum_{k=1}^{n}\ceil*{\frac{\al_k}{q^m}}-\sum_{k=1}^{r}\ceil*{\frac{t\la_k}{q^m}}\deg(P_k)=\rho,\quad \text{where}\quad \rho=\left\{
\begin{array}{ll}
	0, & \text{if }\, \Upsilon=\Gamma,\\
	n-2, & \text{if }\, \Upsilon=\Lambda.
\end{array}\right.
$$  
\end{theorem}

\section{Explicit description of the set $\widehat{\Upsilon}(\Q)$}\label{Section 4}

In this section, using the arithmetic characterization of absolute and relative maximal elements of the generalized Weierstrass semigroup $\widehat{H}(\mathbf{Q})$ given in the previous section, we provide an explicit description of the set $\widehat{\Upsilon}(\mathbf{Q})$. In \cite[Theorem 4.1]{ZZ2026}, Zhang and Zhao provided an explicit description of the minimal generating set $\Gamma(\mathbf{Q})$ via complex constructions of functions in $K(\mathcal{X})$. Our explicit description, which comes from arithmetic methods and the characterization of maximal elements of $\widehat{H}(\mathbf{Q})$ via discrepancies, generalizes this result. 

To provide an explicit description of set $\widehat{\Upsilon}(\Q)$, we define, for each $1 \leq i \leq q^m-1$ and $1 \leq l \leq r$, the functions
\begin{equation}\label{t_beta}
	t_l(i):= (i\la_l)\bmod q^m \quad \text{and} \quad \be(i):=\sum_{k=1}^{r}\ceil*{\frac{i\la_k}{q^m}}\deg(P_k)-1.
\end{equation}

Note that these functions are similar to those employed in \cite{CMS2025} and \cite{CMT2026} to describe the gap sets and generalized Weierstrass semigroups, respectively, in arbitrary Kummer extensions of the rational function field. Moreover, we can utilize them to describe the gap set and the Weierstrass semigroup at any totally ramified place of degree one in the extension $K(\mathcal{X})/K(x)$.

\begin{proposition}\cite[Proposition 3.4 and Corollary 3.6]{ZZ2026}\label{teo_gapset}
	Let $1\leq l \leq r$ be such that $Q_l$ is a rational place. Then
	$$
	G(Q_l)=\{q^mj+t_l(i): 1 \leq i \leq q^m-1, \, 0\leq j \leq \be(i)-1\}
	$$
	and 
	$$
	H(Q_l)=\langle q^m, q^m\be(i)+t_l(i): 1 \leq i \leq q^m-1\rangle.
	$$
\end{proposition}

We can now present our main result, in which we provide an explicit description of the set $\widehat{\Upsilon}(\Q)$ in terms of the functions defined in (\ref{t_beta}). This description is similar to the description of the maximal elements of the generalized Weierstrass semigroup in Kummer extensions given in \cite[Theorem 3.4]{CMT2026}.

\begin{theorem}\label{teo_widehat_Upsilon}
Let $2\leq n \leq q^s$ and $\Q=(Q_{1}, Q_{2}, \dots, Q_{n})$ be an $n$-tuple of distinct totally ramified places of degree one in $K(\cX)/K(x)$. Then
\begin{align*}
\widehat{\Upsilon}(\Q)=&\Bigg\{(q^mj_1+t_{1}(i), \dots,  q^mj_n+t_{n}(i)):
\begin{array}{l}
	1\leq i \leq q^m-1,\,\, j_1, \dots, j_n\in \Z, \\
	j_1+\cdots +j_n=\beta(i)+1-n+\rho
\end{array}
\Bigg\}\\
& \bigcup \Bigg\{(q^mj_1, q^mj_2, \dots, q^mj_n):\, j_1, j_2, \dots, j_n \in \Z, \, \,   j_1+j_2+\cdots+j_n=\rho \Bigg\},
\end{align*}
where 
$$\quad \rho=\left\{
\begin{array}{ll}
0, & \text{if }\, \Upsilon=\Gamma,\\
n-2, & \text{if }\, \Upsilon=\Lambda.
\end{array}\right.
$$  
\end{theorem}
\begin{proof}
Define the set 
\begin{align*}
	\widehat{\Upsilon}=&\Bigg\{(q^mj_1+t_{1}(i), \dots,  q^mj_n+t_{n}(i)):
	\begin{array}{l}
		1\leq i \leq q^m-1,\,\, j_1, \dots, j_n\in \Z, \\
		j_1+\cdots +j_n=\beta(i)+1-n+\rho
	\end{array}
	\Bigg\}\\
	& \bigcup \Bigg\{(q^mj_1, q^mj_2, \dots, q^mj_n):\, j_1, j_2, \dots, j_n \in \Z, \, \,   j_1+j_2+\cdots+j_n=\rho \Bigg\}.
\end{align*}
Let $\negalpha\in \widehat{\Upsilon}$. We will prove that $\negalpha\in \widehat{\Upsilon}(\Q)$. We will analyze two cases:\\ 

\noindent $\diamond$ {\it Case $\negalpha=(q^mj_1+t_1(i), \dots, q^mj_n+t_n(i))$}:\\  
\noindent In this case note that there exists a unique integer $t:=i\in \{0, \dots, q^m-1\}$ such that $q^mj_k+t_k(i)-t\la_k\equiv 0 \bmod{q^m}$ for every $1\leq k \leq n$. Moreover, since $1 \leq t_{k}(i) \leq q^m-1$, we have
\begin{align*}
\sum_{k=1}^{n}\ceil*{\frac{q^mj_k+t_k(i)}{q^m}}-\sum_{k=1}^{r}\ceil*{\frac{t\la_k}{q^m}}\deg(P_k)&=
\sum_{k=1}^{n}\left(j_k+\ceil*{\frac{t_k(i)}{m}}\right)-\sum_{k=1}^{r}\ceil*{\frac{i\la_k}{q^m}}\deg(P_k)\\
&=j_1+\cdots+j_n+n-\sum_{k=1}^{r}\ceil*{\frac{i\la_k}{q^m}}\deg(P_k)\\
&=\be(i)+1+\rho-\sum_{k=1}^{r}\ceil*{\frac{i\la_k}{q^m}}\deg(P_k)\\
&=\rho.
\end{align*}

\noindent$\diamond$ {\it Case $\negalpha=(q^mj_1, \dots, q^mj_n)$}:\\
\noindent Analogously to the previous case, there exists a unique integer $t:=0\in \{0, \dots, q^m-1\}$ such that $q^mj_k-t\la_k\equiv 0 \bmod{q^m}$ for every $1\leq k \leq n$. Furthermore,
$$
\sum_{k=1}^{n}\ceil*{\frac{q^mj_k}{q^m}}-\sum_{k=1}^{r}\ceil*{\frac{t\la_k}{q^m}}\deg(P_k)=j_1+\cdots+j_n=\rho.
$$
Thus, by Theorem \ref{lemma_criterion_Upsilon}, $\negalpha\in \widehat{\Upsilon}(\Q)$ and consequently $\widehat{\Upsilon}\subseteq \widehat{\Upsilon}(\Q)$.\\ 

Conversely, let $\negalpha=(\al_1, \dots, \al_n) \in \widehat{\Upsilon}(\Q)$. From Theorem \ref{lemma_criterion_Upsilon}, there exists a unique $i\in \{0, \dots, q^m-1\}$ such that $\al_k-i\la_k\equiv 0 \bmod{q^m}$ for every $1\leq k\leq n$, and
$$
\sum_{k=1}^{n}\ceil*{\frac{\al_k}{q^m}}-\sum_{k=1}^{r}\ceil*{\frac{i\la_k}{q^m}}\deg(P_k)=\rho.
$$
For each $1\leq k \leq n$, we can write $\al_k=q^mj_k+i_k$, where $j_k\in \Z$ and $0\leq i_k \leq q^m-1$. We will analyze two cases:\\

\noindent$\diamond$ {\it Case $i=0$}:\\
\noindent Here we have that $\al_k\equiv 0 \bmod{q^m}$ for every $1\leq k\leq n$, and thus $i_1=i_2=\dots=i_n=0$. So, 
$$
\rho=\sum_{k=1}^{n}\ceil*{\frac{\al_k}{q^m}}-\sum_{k=1}^{r}\ceil*{\frac{i\la_k}{q^m}}\deg(P_k)=\sum_{k=1}^{n}\ceil*{\frac{q^mj_k}{q^m}}=j_1+j_2+\cdots+j_n.
$$
This implies that $\negalpha=(q^mj_1, q^mj_2, \dots, q^mj_n) \in \widehat{\Upsilon}$.\\

\noindent $\diamond$ {\it Case $i\neq 0$}:\\
\noindent From the relations $\al_k-i\la_k\equiv 0 \bmod{q^m}$ and $\al_k=q^mj_k+i_k$ we obtain that $i_k \equiv i \la_k \bmod{q^m}$ and therefore $i_k=t_k(i)$ for every $1\leq k \leq n$. Thus,
\begin{align*}
\rho&=\sum_{k=1}^{n}\ceil*{\frac{\al_k}{q^m}}-\sum_{k=1}^{r}\ceil*{\frac{i\la_k}{q^m}}\deg(P_k)\\
&=\sum_{k=1}^{n}\ceil*{\frac{q^mj_k+t_k(i)}{q^m}}-\sum_{k=1}^{r}\ceil*{\frac{i\la_k}{q^m}}\deg(P_k)\\
&=\sum_{k=1}^{n}\left(j_k+\ceil*{\frac{t_k(i)}{q^m}}\right)-\sum_{k=1}^{r}\ceil*{\frac{i\la_k}{q^m}}\deg(P_k)\\
&=j_1+j_2+\cdots+j_n+n-\be(i)-1.
\end{align*}
Thus, we conclude that $\negalpha=(q^mj_1+t_1(i), \dots, q^mj_n+t_n(i)) \in \widehat{\Upsilon}$ and  $\widehat{\Upsilon}(\Q)\subseteq \widehat{\Upsilon}$. 
\end{proof}

Note that Theorem \ref{teo_widehat_Upsilon} generalizes the result established in \cite[Theorem 4.1]{ZZ2026}, where the authors determined the set $\Gamma(\mathbf{Q})$. In fact, as an immediate consequence of the previous theorem, we obtain the following corollary.

\begin{corollary}\label{coro_Upsilon}
	Let $2\leq n \leq q^s$ and $\Q=(Q_{1}, Q_{2}, \dots, Q_{n})$ as in Theorem \ref{teo_widehat_Upsilon}. Then
	$$
		\Upsilon(\Q)=\Bigg\{(q^mj_1+t_{1}(i), \dots,  q^mj_n+t_{n}(i)):
		\begin{array}{l}
			1\leq i \leq q^m-1,\,\, j_1, \dots, j_n\in \N_0, \\
			j_1+\cdots +j_n=\beta(i)+1-n+\rho
		\end{array}
		\Bigg\},
	$$
	where 
	$$\quad \rho=\left\{
	\begin{array}{ll}
		0, & \text{if }\, \Upsilon=\Gamma,\\
		n-2, & \text{if }\, \Upsilon=\Lambda.
	\end{array}\right.
	$$  
\end{corollary}

Finally, we present two curves to which we apply our obtained results: the Artin-Mumford curve and a certain curve with many rational places.

\begin{example}[Artin-Mumford curve]\label{ex1}
	Let $p\geq 3$ be a prime number.
	The Artin-Mumford curve over $\fp$ is the curve defined by the affine equation 
	$$
	\cM: \quad (y^p-y)(x^p-x)=1.
	$$
	This curve was introduced to investigate properties unique to curves in positive characteristic. For example, for $p>37$, the Artin-Mumford curve $\cM$ is an example of curve with a large automorphism group (that is, whose size exceeds the Hurwitz bound) and with non-zero $p$-rank, see \cite{D1975}.

	Note that the pole divisor of the rational function $\frac{1}{x^p-x}$ in  $\mathbb{F}_{p}(x)$ is given by
	$$
	\left(\frac{1}{x^p-x}\right)_\infty=P_1+P_2+\dots+P_p,
	$$
	where $P_1, \dots, P_p$ are the places in $\mathbb{F}_{p}(x)$ corresponding to the zeros of $x^p-x$. Therefore, for this case, we have
	$$
	\be(i)=\sum_{k=1}^{p}\ceil*{\frac{i}{p}}-1=p-1\quad \text{for every}\quad 1\leq i \leq p-1.
	$$
	
	Now, for $2\leq n\leq p$, let $Q_1, Q_2, \dots, Q_n$ be places in $\mathbb{F}_p(\cM)$ associated to te zeros of $x^p-x$, and let $\Q=(Q_1, Q_2, \dots, Q_n)$. Thus, from Theorem \ref{teo_widehat_Upsilon}, we obtain 
	\begin{align*}
		\widehat{\Upsilon}(\Q)=&\Bigg\{(pj_1+i, \dots,  pj_n+i):
		\begin{array}{l}
			1\leq i \leq p-1,\,\, j_1, \dots, j_n\in \Z, \\
			j_1+\cdots +j_n=p-n+\rho
		\end{array}
		\Bigg\}\\
		& \bigcup \Bigg\{(pj_1, \dots, pj_n):\, j_1, \dots, j_n \in \Z, \, \,   j_1+\cdots+j_n=\rho \Bigg\},
	\end{align*}
	where $\rho=0$ if $\Upsilon=\Gamma$ and $\rho=n-2$ if $\Upsilon=\Lambda$.
	
\end{example}

\begin{example}[Curve with many rational places]\label{ex2}
	Consider the curve over $\mathbb{F}_{121}$ defined by the affine equation
	$$
	\cY:\quad y^{11}+y=\frac{(x^2+1)^2}{x^2}.
	$$
	This curve was introduced in \cite[Example 4.12]{GMQ2023}, and it is the curve of genus $20$ over $\mathbb{F}_{121}$ with the highest number of rational places recorded in the manYPoints Table \cite{MP2009}, having exactly $442$ rational places.
	
	The pole divisor of the rational function $\frac{(x^2+1)^2}{x^2}$ in  $\mathbb{F}_{121}(x)$ is given by
	$$
	\left(\frac{(x^2+1)^2}{x^2}\right)_\infty=2P_0+2P_\infty,
	$$
	where $P_0$ and $P_\infty$ are the places in $\mathbb{F}_{121}(x)$ corresponding to the zero and pole of $x$, respectively. Therefore, for this curve, we have that
	$$
	\be(i)=2\ceil*{\frac{2i}{11}}-1\quad \text{for every}\quad 1\leq i\leq 10.
	$$

Let $Q_0, Q_\infty$ be places in  $\mathbb{F}_{121}(\cY)$ such that $Q_0|P_0$ and  $Q_\infty|P_\infty$. Then, for $\Q=(Q_0, Q_\infty)$, we have
\begin{align*}
	\widehat{\Upsilon}(\Q)=&\Bigg\{(11j_1+(2i)\bmod{11},  11j_n+(2i)\bmod{11}):
	\begin{array}{l}
		1\leq i \leq 10,\,\, j_1, j_2\in \Z, \text{ and} \\
		j_1+j_2=2\ceil*{\frac{2i}{11}}-2
	\end{array}
	\Bigg\}\\
	& \bigcup \Bigg\{(11j_1, 11j_2):\, j_1, j_2\in \Z, \, \,   j_1+j_2=0 \Bigg\}.
\end{align*}
In particular,
\begin{align*}
	\Gamma(\Q)=\Lambda(\Q)=\{&  (8, 8), (20, 20), (6, 6), (31, 9), (12, 12), (1, 23), (2, 2), (14, 14), (3, 25),\\
	& (10, 10), (9, 31), (4, 4), (7, 29), (27, 5), (18, 18), (23, 1), (5, 27), (25, 3),\\
	& (16, 16), (29, 7) \}.
\end{align*}
Furthermore, from Theorem \ref{coro_G0(Q)}, we can determine the set of pure gaps. In this case, we obtain 
\begin{align*}
	G_0(\Q)=\{& (9, 3), (3, 8), (9, 9), (7, 20), (6, 5), (12, 3), (3, 3), (5, 12), (16, 5),
	(7, 9), (4, 3),\\
	& (1, 5), (12, 5), (7, 1), (1, 4), (9, 12), (14, 3), (20, 9), (20, 3), (5, 1), (1, 20), (7, 10),\\
	& (5, 8), (16, 1), (1, 12), (3, 6), (5, 3), (18, 5),
	(3, 1), (1, 10), (3, 12), (5, 5), (1, 16),\\
	& (5, 9), (18, 3), (5, 20), (7, 7), (5, 7), (1, 18), (3, 20), (1, 6), (10, 7), (20, 5), (14, 9),\\
	& (9, 10), (12, 7), (2,
	1), (1, 3), (3, 7), (9, 16), (3, 10), (10, 3), (1, 14), (7, 12), (5, 14),\\
	& (16, 7), (14, 5), (9, 14), (18, 7), (8, 1), (8, 3), (7, 5), (5, 18), (8, 5), (3, 16),
	(18, 9),\\
	& (16, 3), (5, 10), (3, 14), (20, 7), (3, 5), (12, 9), (1, 1), (8, 7), (4, 1), (5, 6), (3, 9),\\
	& (12, 1), (16, 9), (3, 18), (9, 18), (14, 1), (7, 3), (9,
	20), (10, 5), (10, 9), (1, 8),\\
	& (1, 7), (7, 8), (5, 16), (7, 14), (20, 1), (1, 2), (3, 4), (18, 1), (1, 9), (10, 1), (14, 7),\\
	& (9, 7), (6, 1), (9, 5), (7, 16), (9, 1), (6, 3), (7, 18)\}.
\end{align*}
\end{example}

\bibliographystyle{abbrv}

\bibliography{generalizedws} 

\end{document}